\documentclass[preprint,11pt]{elsarticle}

\usepackage[letterpaper,margin=1in]{geometry}
\usepackage{amssymb}
\usepackage{amsthm}
\usepackage{amsmath}
\usepackage{mathdots}
\usepackage{color}
\usepackage{bm}

\newtheorem{thm}{Theorem}[section]
\newtheorem{lem}[thm]{Lemma}

\theoremstyle{definition}
\newtheorem{definition}[thm]{Definition}

\begin{document}

\begin{frontmatter}



\title{On the validity of the local Fourier analysis} 


\author[UniZar]{Carmen Rodrigo\corref{cor1}}\ead{carmenr@unizar.es}
\address[UniZar]{Departamento de Matem\'{a}tica Aplicada, IUMA,
Universidad de Zaragoza, Zaragoza, Spain}
\author[CWI]{Francisco J. Gaspar}\ead{F.J.Gaspar@cwi.nl}
\address[CWI]{CWI, Centrum Wiskunde \& Informatica, Science Park 123, 1090 Amsterdam, The Netherlands}
\author[PennState,BAS]{Ludmil T. Zikatanov}\ead{ludmil@psu.edu}
\address[PennState]{Department of Mathematics, The Pennsylvania State University, 
  University Park, Pennsylvania, 16802, USA}
\address[BAS]{Institute for Mathematics and 
  Informatics, Bulgarian Academy of Sciences, Sofia, Bulgaria}
\cortext[cor1]{Corresponding author. Tel.: +34 976762148; E-mail address: carmenr@unizar.es (Carmen Rodrigo)}

\begin{abstract}
  Local Fourier analysis (LFA) is a useful tool in predicting the
  convergence factors of geometric multigrid methods (GMG). As is well
  known, on rectangular domains with periodic boundary conditions this
  analysis gives the \emph{exact} convergence factors of
  such methods. In this work, using the Fourier method, we extend
  these results by proving that such analysis yields the exact
  convergence factors for a wider class of problems.
\end{abstract}

\begin{keyword}
Local Fourier analysis \sep multigrid \sep Fourier method



\end{keyword}

\end{frontmatter}


\section{Introduction}
\label{sec:intro}
The local Fourier analysis (LFA) introduced by A. Brandt \cite{Bra77},
is a tool which provides realistic quantitative estimates of the
asymptotic convergence factors of the GMG algorithms.  For
discretizations of partial differential equations, the traditional LFA
is based on a discrete Fourier transform and is accurate if the
influence of the boundary conditions is negligible.  In fact, it is
well known (see \cite{Stu1982, Brandt1994}), that for model
problems on rectangular domains and with periodic boundary
conditions this analysis gives the exact convergence rate of GMG.

In this work we focus on the question whether the LFA can be made
rigorous for a wider class of problems with boundary
conditions that are not necessarily periodic. We answer to this
question positively.  Our approach relies on the embedding of the
model problem into a periodic problem. Similar ideas have also been
explored in works on circulant preconditioners for elliptic
problems~\cite{Lirkov1994, Lirkov1997} and also for preconditioning
the indefinite Helmholtz equation~\cite{ludmil}. We introduce a class
of operators called LFA-compatible operators here and prove that for
such operators the LFA gives the exact multigrid convergence
factors. Our studies include the Dirichlet, the Neumann and the mixed
boundary condition problem for a constant coefficient,
reaction-diffusion equation on a $d$-dimensional tensor product grid.

\section{Preliminaries}

\subsection{The Dirichlet problem and its discretization}
We consider a reaction-diffusion problem in $d$ spatial dimensions on
the domain $\Omega^D =(0,1)^d$,
\begin{equation}\label{dD_model_problem_D}
-\Delta u(\bm{x}) + c u(\bm{x}) = f(\bm{x}), \ \ \bm{x} \in \Omega^D, 
\quad\mbox{and}\quad 
u(\bm{x}) = 0, \ \ \bm{x} \in \partial \Omega^D,
\end{equation}
where $c>0$ is a constant. First, let us consider the simplest case
when $d=1$ (one dimensional problem). The computational domain then is
the interval $\Omega^D =(0,1)$ and the corresponding two-point
boundary value problem~\eqref{dD_model_problem_D} is:
\begin{equation}\label{1D_model_problem_D}
-u''(x) + c u(x) = f(x), \ \ x\in \Omega^D,\quad u(0) = u(1) = 0.
\end{equation}
For $d=1$, we introduce a uniform grid
$\Omega_h^D = \{x_k=kh\}_{k = 0}^n$, with step size $h=1/n$,
$n\in \mathbb{N}$ and we discretize this problem by the standard
central difference scheme. As a result, we obtain the linear system of
algebraic equations with tri-diagonal matrix:
\begin{equation}\label{system_eqs_D}
A_h^D \bm{u} = \bm{f} \ \hbox{where} \ 
A_h^D = T_h^D +c I_{n-1}\in \mathbb{R}^{(n-1)\times (n-1)},
\end{equation}
where $\bm{u} = \left(u_1, \ldots,
 u_{n-1}\right)^T$, $\bm{f} = \left(f_1, \ldots, f_{n-1}\right)^T$,
$I_{n-1}\in \mathbb{R}^{(n-1)\times (n-1)}$ is the identity matrix, and
\begin{equation}\label{T_D}
  T_h^D = \frac{1}{h^2}\ \operatorname{diag}(-1,2,-1) 
\in \mathbb{R}^{(n-1)\times (n-1)}.
\end{equation}

This was the simple, but very important, one dimensional case. In the case of 
higher spatial dimensions and on a uniform grid with the same
step size $h=1/n$ in all the directions the linear
systems are written in compact form by using the standard tensor
product $\otimes$ for matrices. We recall the following properties of
the tensor product
\begin{equation}\label{tensor_product}
(X + Y) \otimes Z = (X \otimes Z) + (Y \otimes Z), \quad (X_1 \otimes
X_2)(Y_1 \otimes Y_2) = (X_1Y_1 \otimes X_2 Y_2).
\end{equation}
We further denote the $k$-th tensor power of a matrix $X$ by
$X^{\otimes k} = \underbrace{X \otimes \ldots \otimes X}_{k}$.
Finally, let us note that the generalization to different step sizes in
different directions is straightforward.

With this notation, the standard second order central difference
scheme for discretization of the Dirichlet
problem~\eqref{dD_model_problem_D} results in the linear system 
\begin{equation}\label{system_eqs_dD}
A_h^D \bm{u} = \bm{f}, \quad   A_h^D = \sum_{j=1}^d  \left(I_{n-1}^{\otimes(j-1)} \otimes T_h^D \otimes I_{n-1}^{\otimes(d-j)}\right) + c I_{n-1}^{\otimes d} \in \mathbb{R}^{(n-1)^d\times (n-1)^d}.
\end{equation}

\subsection{A periodic problem}\label{s:periodic} 
We now consider a finite difference discretization on a grid with step size $h = 1/n$ of a periodic
problem on $\Omega^P=(0,2)$:
\begin{equation}\label{system_eqs_P}
A_{h}^P \widetilde{\bm{u}}= \widetilde{\bm{f}}, \quad \hbox{where} \quad
A_{h}^P = T_h^P+ c I_{N}\in \mathbb{R}^{N\times N},
\end{equation}
with $N=2n$ and
$T_{h}^P = \frac{1}{h^2}\ \operatorname{diag}(-1,2,-1) -
\bm{e}_1^N(\bm{e}^N_N)^T - \bm{e}^N_N(\bm{e}^N_1)^T \in \mathbb{R}^{N\times N}$.  Here, we have denoted
$\widetilde{\bm{u}} = \left(\widetilde{u}_1, \ldots,
  \widetilde{u}_{N}\right)^T$, 
$\widetilde{\bm{f}} = \left(\widetilde{f}_1, \ldots,
  \widetilde{f}_{N}\right)^T$, and $\bm{e}_k^m$ is the $k$-th canonical Euclidean basis vector in $\mathbb{R}^{m}$.  Finally, let us point out that by a
periodic problem here we mean the problem~\eqref{1D_model_problem_D} defined on $\Omega^P$
with boundary conditions $u(0)-u(2)=u'(0)-u'(2)=0$. 

The extension to higher dimension $d>1$ is obvious and we have
the linear system $A_h^P \widetilde{\bm{u}} = \widetilde{\bm{f}}$,
with
\begin{equation}\label{system_eqs_dP}
A_h^P = \sum_{j=1}^d \left(I_N^{\otimes(j-1)} \otimes T_h^P \otimes I_N^{\otimes(d-j)}\right)
+ c I_N^{\otimes d} \in \mathbb{R}^{N^d\times N^d}.
\end{equation}

\subsection{Relation between the Dirichlet and the periodic problem}
Our goal now is to describe how the discretized Dirichlet problem
relates to the periodic problem defined in section~\ref{s:periodic}. To begin, we
consider the 1-dimensional case given in~\eqref{system_eqs_D}
and we define the {\it odd extension operator} 
as the linear operator
$E_{o,h}:\ \mathbb{R}^{n-1} \ \rightarrow \ \mathbb{R}^N$, $N=2n$ such
that
\begin{equation}\label{extension_op}
E_{o,h} \bm{e}_i^{n-1} = \bm{e}_i^N - \bm{e}_{N-i}^N, \quad i=1,\ldots,n-1.
\end{equation} 
The \emph{restriction operator} $R_{o,h}$ is defined as
$R_{o,h} = \frac{1}{2} E_{o,h}^T$. 
It is easy to see that the following relations hold in the one dimensional case:
$R_{o,h} E_{o,h}= I_{n-1}$, and 
$E_{o,h} R_{o,h} \bm{u} = \bm{u}$, for all 
$\bm{u}\in\operatorname{range} (E_{o,h})$. 
Notice also that
$\operatorname{range} (E_{o,h}) = \{\bm{u}\in \mathbb{R}^N \ | u_n = u_N
= 0, u_{j} = -u_{N-j}, j=1,\ldots,n-1\}$
and $\widetilde{\bm{f}} = E_{o,h} \bm{f}$.
For $d>1$ the restriction and extensions are 
$R_{o,h}^{\otimes d}$ and $E_{o,h}^{\otimes d}$ and we have:
\begin{equation}\label{property12}
R_{o,h}^{\otimes d}  E_{o,h}^{\otimes d} = I_{n-1}^{\otimes d}, 
\quad\mbox{and}\quad 
E_{o,h}^{\otimes d} R_{o,h}^{\otimes d} \bm{u} = \bm{u}, 
\quad 
\mbox{for all}\quad \bm{u}\in \operatorname{range} (E_{o,h}^{\otimes d}). 
\end{equation}

\subsubsection{LFA-compatibility} \label{s:lfa-compatibility}


We now clarify the relation between the Dirichlet and the periodic
problem. We begin with a very general definition of LFA-compatibility.
\begin{definition}
  Let $R_{o,h}$ and $E_{o,h}$ be operators
  satisfying~\eqref{property12}. We say that the pair of operators
  $(M_h^D,M_h^P)$ is an
  \emph{LFA-compatible} pair if and only if $M_h^D=R_{o,h} M_h^P E_{o,h}$ and
  $M_h^P \bm{v} \in \operatorname{range} (E_{o,h})$ for
  all $\bm{v} \in \operatorname{range} (E_{o,h})$.
\end{definition}
The LFA-compatibility is, in some sense, the minimal requirement which
allows for building relations between solutions to a periodic and
the corresponding Dirichlet problems, or the iterates constructed in an iterative method
for these problems. In a more abstract setting, the operators $M_h^P$
and $M_h^D$ do not have to be a periodic or a Dirichlet problem, they only
need to be connected via a compatibility relation based on operators
$E_{o,h}$ and $R_{o,h}$ satisfying the relations  
in~\eqref{property12}. In the following, however, we only use 
$E_{o,h}$ and $R_{o,h}$ as defined above.

Now we prove several results, which follow directly from the
definition of LFA-compatibility.
\begin{lem}\label{Th_relation}
Let $A_h^D$ and $A_h^P$ be the coefficient matrices related to the Dirichlet and periodic problems. Then, $(A_h^D,A_h^P)$ is an LFA-compatible pair.
\end{lem}
\begin{proof}
The standard properties of the tensor product imply that
\begin{eqnarray*}
R_{o,h}^{\otimes d} A_h^P  E_{o,h}^{\otimes d} &=& R_{o,h}^{\otimes d} \left(\sum_{j=1}^d \left(I_N^{\otimes(j-1)} \otimes T_h^P \otimes I_N^{\otimes(d-j)}\right)
+ c I_N^{\otimes d}\right) E_{o,h}^{\otimes d} \\
& = & R_{o,h}^{\otimes d}\left(\sum_{j=1}^d  \left(E_{o,h}^{\otimes(j-1)} \otimes T_h^P E_{o,h} \otimes E_{o,h}^{\otimes(d-j)}\right)
+ c E_{o,h}^{\otimes d}\right) \\
& = & \sum_{j=1}^d  \left(I_{n-1}^{\otimes(j-1)} \otimes R_{o,h}T_h^P E_{o,h} \otimes I_{n-1}^{\otimes(d-j)}\right)
+ c  I_{n-1}^{\otimes d}.
\end{eqnarray*}
Further, taking into account that  $R_{o,h} T_h^P E_{o,h} = T_h^D$, we also have $A_h^D = R_{o,h}^{\otimes d}  A_h^P  E_{o,h}^{\otimes d}$.  
If $\bm{u} \in \operatorname{range} (E_{o,h}^{\otimes d})$, then there exists 
$\bm{v} \in \mathbb{R}^{(n-1)^d}$ such that
$\bm{u} = E_{o,h}^{\otimes d}\bm{v}$ and we have 
\begin{eqnarray*}
A_h^P \bm{u} &=& \left(\sum_{j=1}^d I_N^{\otimes(j-1)} \otimes T_h^P \otimes I_N^{\otimes(d-j)}\right) E_{o,h}^{\otimes d}\bm{v}+ c I_N^{\otimes d}E_{o,h}^{\otimes d}\bm{v} \\
&=&  \left(\sum_{j=1}^d E_{o,h}^{\otimes(j-1)} \otimes T_h^P E_{o,h} \otimes E_{o,h}^{\otimes(d-j)}\right)  \bm{v}+ c E_{o,h}^{\otimes d}\bm{v}.
\end{eqnarray*}
A straightforward computation shows that
$T_h^P \bm{u} \in \operatorname{range}(E_{o,h})$
for any $\bm{u} \in \operatorname{range}(E_{o,h})$, and this completes the proof.
\end{proof}
\begin{lem}\label{th_sol_DP}
If $\bm{u}$ satisfies $A_h^D \bm{u} = \bm{f}$, 
then  $A_h^P (E_{o,h}^{\otimes d} \bm{u}) = E_{o,h}^{\otimes d} \bm{f}$.
\end{lem}
\begin{proof}
Using that $A_h^D = R_{o,h}^{\otimes d}  A_h^P  E_{o,h}^{\otimes d}$, we have that $R_{o,h}^{\otimes d}  A_h^P  E_{o,h}^{\otimes d} \bm{u} = \bm{f}.$ Applying $E_{o,h}^{\otimes d}$ on the left and taking into account that $A_h^P E_{o,h}^{\otimes d} \bm{u} \in \operatorname{range} (E_{o,h}^{\otimes d})$, completes the proof. 
\end{proof}
\begin{thm}\label{inv_invariance_d}
The pair $((A_h^D)^{-1},(A_h^P)^{-1})$ is LFA-compatible.
\end{thm}
\begin{proof}
  We consider $\bm{f} \in \operatorname{range}(E_{o,h}^{\otimes d})$. Then, there exists $\bm{g}\in \mathbb{R}^{{(n-1)}^d}$ such that $E_{o,h}^{\otimes d} \bm{g} = \bm{f}$. If $\bm{u} = (A_h^D)^{-1}\bm{g}$, by using Lemma~\ref{th_sol_DP} we have that $E_{o,h}^{\otimes d}\bm{u}=(A_h^P)^{-1}\bm{f}$, which implies that $(A_h^P)^{-1}\bm{f}\in \operatorname{range}(E_{o,h}^{\otimes d})$. 
Next, again from Lemma~\ref{th_sol_DP}, it follows that 
if $\bm{u}=(A_h^D)^{-1}\bm{f}$, then 
  $(A_h^P)^{-1} E_{o,h}^{\otimes d} \bm{f} = E_{o,h}^{\otimes d}\bm{u}$. Hence, 
  $R_{o,h}^{\otimes d}(A_h^P)^{-1} E_{o,h}^{\otimes d} \bm{f} =
  R_{o,h}^{\otimes d} E_{o,h}^{\otimes d} \bm{u} = \bm{u}$ and the 
proof is complete. 
\end{proof}

\section{Linear iterative methods and multigrid}
Let us consider a general stationary iterative method for the Dirichlet and the periodic problems: 
\begin{equation}
\bm{u}^{k+1} = \bm{u}^k + B_h^D(\bm{f}-A_h^D\bm{u}^k), \quad
\widetilde{\bm{u}}^{k+1} = \widetilde{\bm{u}}^k + B_h^P(\widetilde{\bm{f}}-A_h^P\widetilde{\bm{u}}^k),\label{iterative_d}
\end{equation}
where $B_h^{D,P}$ are linear operators (called iterators). 
We have the following theorem which shows that
the LFA-compatibility of the iterators provides a relation between the
iterates. 
\begin{thm}\label{th_iterdD}
  Let $(B_h^D,B_h^P)$ be an LFA-compatible pair and
  $\widetilde{\bm{f}} = E_{o,h}^{\otimes d} \bm{f}$. If
  $\widetilde{\bm{u}}^0 = E_{o,h}^{\otimes d} \bm{u}^0$, then
  $\widetilde{\bm{u}}^{k} = E_{o,h}^{\otimes d} \bm{u}^{k}, \;
  k=1,2,\ldots$
\end{thm}
\begin{proof}
  We prove the result by showing that if
  $\widetilde{\bm{u}}^k = E_{o,h}^{\otimes d} \bm{u}^k$ then
  $\widetilde{\bm{u}}^{k+1} = E_{o,h}^{\otimes d}
  \bm{u}^{k+1}$. Clearly, from \eqref{iterative_d}, and the fact that
  $\widetilde{\bm{u}}^k = E_{o,h}^{\otimes d} \bm{u}^k$, we have
  $\widetilde{\bm{u}}^{k+1} = E_{o,h}^{\otimes d} \bm{u}^k +
  B_h^P(E_{o,h}^{\otimes d} \bm{f}-A^P_h E_{o,h}^{\otimes d}
  \bm{u}^k)$. Next, we use
  Lemmata~\ref{Th_relation}--\ref{inv_invariance_d} to obtain that,
\[
\widetilde{\bm{u}}^{k+1} = E_{o,h}^{\otimes d}\bm{u}^k + B_h^P(E_{o,h}^{\otimes d} \bm{f}-E_{o,h}^{\otimes d}
R_{o,h}^{\otimes d} A^P_h E_{o,h}^{\otimes d} \bm{u}^k) =  E_{o,h}^{\otimes d}\bm{u}^k + B_h^P E_{o,h}^{\otimes d}(\bm{f}-A^D_h \bm{u}^k).
\]
Since,
$E_{o,h}^{\otimes d}(\bm{f}-A^D_h \bm{u}^k)\in \operatorname{range}
(E_{o,h}^{\otimes d})$, and 
$B_h^P E_{o,h}^{\otimes d}(\bm{f}-A^D_h \bm{u}^k)\in
\operatorname{range} (E_{o,h}^{\otimes d})$ 
we have that 
\[
\widetilde{\bm{u}}^{k+1} = E_{o,h}^{\otimes d} \bm{u}^k + E_{o,h}^{\otimes d} R_{o,h}^{\otimes d} B_h^P E_{o,h}^{\otimes d}(\bm{f}-A^D_h \bm{u}^k).
\]
Finally, we use that $(B_h^D,B_h^P)$ is an LFA-compatible pair to
obtain that
$\widetilde{\bm{u}}^{k+1} = E_{o,h}^{\otimes d} (\bm{u}^k +
B_h^D(\bm{f}- A_h^D \bm{u}^k)) = E_{o,h}^{\otimes d} \bm{u}^{k+1}$
which is what we wanted to show.
\end{proof}

\subsection{Two grid methods}
We now consider the two-grid and multigrid methods. We begin by
defining the coarse grids for the Dirichlet and periodic problems in
one spatial dimension ($d=1$). In a standard fashion, we define
\[
\Omega_{2h}^D = \{x_i = 2ih \ | \ i = 0,\ldots, n/2\},
\quad \mbox{and}\quad \Omega_{2h}^P = \{x_i = 2ih \ | \ i = 0,\ldots, n\}.
\]
We denote by $\mathcal{G}(\Omega_h^{D,P})$,
$\mathcal{G}(\Omega_{2h}^{D,P})$ the subspaces of grid-functions
defined on $\Omega_h^{D,P}$ and $\Omega_{2h}^{D,P}$, respectively.  On
such coarse grid, we also define $A_{2h}^{D,P}$ by~\eqref{system_eqs_dD} but
with $2h$ instead of $h$. The extension to higher spatial dimensions is done using standard tensor products of grids and operators. 

We now consider the two-grid algorithms, which are linear iterative methods already defined in~\eqref{iterative_d} with special iterators  
$B_{TG}=B_{TG}^{D,P}$ as follows:
\begin{equation}\label{two-grid}
B_{TG} = \left(I-(I-I_{2h,h}(A_{2h})^{-1}I_{h,2h}A_h)(I-S_hA_h)\right)(A_h)^{-1}, 
\end{equation}
In~\eqref{two-grid} all operators change depending on whether we
consider Dirichlet or periodic problem, namely, we have $A_h^D$,
$A_h^P$, $I_{2h,h}^D$, $I_{2h,h}^P$, etc. Here, $S_h^{D,P}$ are
relaxation (smoothing) operators,
$I_{h,2h}^{D,P}: \mathcal{G}(\Omega_h^{D,P})\rightarrow
\mathcal{G}(\Omega_{2h}^{D,P})$ are the restriction operators and
$I_{2h,h}^{D,P}:\mathcal{G}(\Omega_{2h}^{D,P}) \rightarrow
\mathcal{G}(\Omega_{h}^{D,P})$ are the prolongation operators.  To prove
the main result, we need to introduce LFA-compatible restriction and
prolongation operators. We say that  
the pairs 
$(I^D_{2h,h},I^P_{2h,h})$ and $(I^D_{h,2h},I^P_{h,2h})$ are \emph{LFA-compatible} if and only if
\begin{eqnarray}\label{h-2h-h-compatible}
I_{h,2h}^D = R_{o,2h} I_{h,2h}^P E_{o,h}, \quad 
I_{h,2h}^P v \in \operatorname{range} (E_{o,2h}), 
\quad\mbox{for all} \quad v \in \operatorname{range} (E_{o,h}),\\
I_{2h,h}^D = R_{o,h} I_{2h,h}^P E_{o,2h}, \quad
I_{2h,h}^P v \in \operatorname{range} (E_{o,h})
\quad\mbox{for all} \quad v \in \operatorname{range} (E_{o,2h}).
\end{eqnarray}

The multigrid iterator is obtained from the two grid by recursion, namely,
\begin{equation}\label{m-grid}
B_{h} = \left(I-(I-I_{2h,h}B_{2h}I_{h,2h}A_h)(I-S_hA_h)\right)(A_h)^{-1}, 
\end{equation}
where $B_{nh}= A_{nh}^{-1}$ for both the Dirichlet and the periodic problem.  

We have the following theorem, showing that the iterations via two grid are related.
\begin{thm}\label{multigrid_eq} 
If $(A_h^D,A_h^P)$, $((A_{2h}^D)^{-1},(A_{2h}^P)^{-1})$, $(S_h^D,S_h^P)$, $(I^D_{2h,h},I^P_{2h,h})$, $(I^D_{h,2h},I^P_{h,2h})$ are LFA compatible,  then $(B_{h}^D,B_{h}^P)$ is
LFA-compatible.
\end{thm}
\begin{proof}
  We prove this theorem for the case $d=1$ only and $B_h=B_{TG}$ as
  the general case follows from recursive application of this argument
  and the properties of tensor product listed earlier.
\begin{eqnarray*}
R_{o,h} B_{TG}^P E_{o,h} &=& R_{o,h} \left(I-(I-I_{2h,h}^P(A^P_{2h})^{-1}I_{h,2h}^PA_h^P)E_{o,h} R_{o,h} (I-S_h^PA_h^P)\right)E_{o,h} R_{o,h} (A_h^P)^{-1} E_{o,h} \\
& = & \left(I-(I-R_{o,h} I_{2h,h}^P(A^P_{2h})^{-1}I_{h,2h}^PA_h^PE_{o,h})  (I-S_h^DA_h^D)\right)  (A_h^D)^{-1}. 
\end{eqnarray*}
Moreover, because of the invariant properties it follows that
\begin{eqnarray*}
R_{o,h} I_{2h,h}^P(A^P_{2h})^{-1}I_{h,2h}^PA_h^PE_{o,h} = (R_{o,h} I_{2h,h}^P E_{2h}) (R_{2h} (A^P_{2h})^{-1} E_{2h})
(R_{2h} I_{h,2h}^P E_{o,h}) (R_{o,h}  A_h^PE_{o,h}).
\end{eqnarray*}
By using the properties in the assumptions in the theorem we have that $B_{TG}^D = R_{o,h} B_{TG}^P E_{o,h}$.
The invariant property of $B_{TG}^P$ follows from the invariant properties of all the operators involved in the two-grid method. 
\end{proof}
\section{Examples and extensions} 
The compatibility result in Theorem~\ref{multigrid_eq} shows that the
LFA, which is strictly justified for periodic problems, provides rigorous
results also for the Dirichlet problems. Of course, this is for particular choices of $S_h$, $I_{h,2h}$ and the rest of the involved operators. LFA-compatible smoothers include the weighted Jacobi method, the Red-Black Gauss-Seidel, line relaxation methods and polynomial smoothers. The frequently used inter-grid transfer operators full-weighting and bilinear interpolation are LFA-compatible restriction and prolongation operators, respectively. Therefore, multigrid methods based on these components applied to problems with Dirichlet boundary conditions can be analyzed rigorously by LFA. 

Moreover, problems with other boundary conditions can also be put into this framework. For example, all the results presented in this work are easily reproduced for the pure Neumann problem by using an even extension operator instead the odd extension operator. Problems
with mixed boundary conditions can also be included in this framework by using an even extension operator followed by an odd extension operator.

We conclude that for a wide range of multigrid components and for problems with other boundary conditions than the periodic ones, the LFA provides rigorous asymptotic multigrid convergence factors.

\section*{Acknowledgements}

The work of F.~J. Gaspar is supported by the European Union's
Horizon 2020 research and innovation programme under the Marie
Sklodowska-Curie grant agreement NO 705402, POROSOS. The research of
C.~Rodrigo is supported in part by the Spanish project FEDER /MCYT
MTM2016-75139-R and the DGA (Grupo consolidado PDIE). L.~Zikatanov is supported in part by NSF DMS-1522615 and DMS-1720114.



\section*{\refname}
\bibliographystyle{elsarticle-num} 
\bibliography{report_LFA}

%
%
%
%
\end{document}